\documentclass[11pt,reqno]{amsart}

\usepackage{amsmath}
\usepackage{rotating}
\usepackage{amsthm}
\usepackage{amssymb}
\usepackage{amsfonts}
\usepackage{amsxtra}
\usepackage{fullpage}
\usepackage{amscd}
\usepackage{mathrsfs}
\usepackage[all]{xy}
\usepackage{enumerate}
\usepackage[mathscr]{eucal}
\usepackage{verbatim}
\usepackage{pdflscape}

\numberwithin{equation}{section}

\theoremstyle{definition}
\newtheorem{theorem}[equation]{Theorem}
\newtheorem{lemma}[equation]{Lemma}
\newtheorem{proposition}[equation]{Proposition}
\newtheorem{corollary}[equation]{Corollary}

\newtheorem{definition}[equation]{Definition}

\newtheorem{question}[equation]{Question}

\newtheorem{remark}[equation]{Remark}

\makeatletter
\renewcommand{\subsection}{\@startsection{subsection}{2}{0pt}{-3ex
plus -1ex minus -0.2ex}{-2mm plus -0pt minus
-2pt}{\normalfont\bfseries}} \makeatother

\newcommand{\Q}{\mathbf{Q}}
\newcommand{\R}{\mathbf{R}}
\newcommand{\C}{\mathbf{C}}

\newcommand{\Id}{\operatorname{Id}}

\newcommand{\Hilb}{\operatorname{Hilb}}
\newcommand{\Stab}{\operatorname{Stab}}
\newcommand{\Hom}{\operatorname{Hom}}

\newcommand{\Sp}{\mathsf{Sp}}
\newcommand{\SL}{\mathsf{SL}}
\newcommand{\GL}{\mathsf{GL}}

\newcommand{\Spec}{\operatorname{\mathsf{Spec}}}

\newcommand{\Cs}{\C^{\times}}
\newcommand{\ResH}{\mathsf{Y}_H}
\newcommand{\QuotH}{\mathsf{X}_H}
\newcommand{\ds}{\dots}
\newcommand{\Irr}{\mathrm{Irr}}
\newcommand{\biDih}{\mathsf{D}}
\newcommand{\Dih}{\mathbf{D}}
\newcommand{\cyc}{\mathsf{C}}
\newcommand{\tetra}{\mathsf{T}}
\newcommand{\octa}{\mathsf{O}}
\newcommand{\isohedral}{\mathsf{I}}
\renewcommand{\o}{\otimes}
\newcommand{\bi}{\mathbf{i}}
\newcommand{\bj}{\mathbf{j}}
\newcommand{\bk}{\mathbf{k}}
\newcommand{\bzeta}{\boldsymbol{\zeta}}
\newcommand{\bomega}{\boldsymbol{\omega}}
\newcommand{\Quat}{\mathbf{H}}
\newcommand{\reg}{\mathrm{reg}}

\newcommand{\Fix}{\mathrm{Fix}}

\begin{document}

\title{On the (non)existence of symplectic resolutions for imprimitive symplectic reflection groups} 


\author{Gwyn Bellamy}

\address{School of Mathematics and Statistics, University Gardens, University of Glasgow, Glasgow, G12 8QW, UK}

\email{gwyn.bellamy@glasgow.ac.uk}

\author{Travis Schedler}

\address{U. Texas at Austin, Math Dept, RLM 8.100, 2515 Speedway Stop
  C1200, Austin, TX 78712, US}

\email{trasched@gmail.com}



\begin{abstract}
  We study the existence of symplectic resolutions of quotient
  singularities $V/G$ where $V$ is a symplectic vector space and $G$
  acts symplectically.  Namely, we classify the symplectically
  irreducible and imprimitive groups, excluding those of the form $K
  \rtimes S_2$ where $K < \SL_2(\C)$,
  for which the corresponding quotient
  singularity admits a projective symplectic resolution.  As a
  consequence, for $\dim V \neq 4$, we classify all quotient
  singularities $V/G$ admitting a projective symplectic resolution
  which do not decompose as a product of smaller-dimensional quotient
  singularities, except for at most four explicit singularities, that occur in dimensions at most $10$, for whom the question of existence remains open.
\end{abstract}
\maketitle


\section{Introduction}

Symplectic quotient singularities have been intensively studied over
the past decade, due to their rich geometric structures, as
illustrated by the symplectic McKay correspondence \cite{SympMcKay},
and their key role in the geometric representation theory of
symplectic reflection algebras. One of the key questions regarding
the geometry of symplectic quotient singularities is whether there
exist symplectic resolutions of the singularity. In this paper, we
classify all irreducible quotient singularities of dimension not equal
to four which admit a projective symplectic resolution, excluding four exceptional 
cases. The main step is to prove that, in dimension four, a large
class of singularities do not admit a resolution. 

More precisely, we classify all symplectically imprimitive and
irreducible symplectic reflection groups (excluding the groups $K
\rtimes S_2 < \Sp_4(\C)$ where $K < \SL_2(\C)$) whose corresponding
quotient singularity admits a projective symplectic resolution.  This
is an important step in an ongoing program to completely classify all
finite subgroups $G$ of $\mathsf{Sp}(V)$ such that $V/G$ admits a
symplectic resolution.

In order to state our main result we introduce some notation. Let $(V,
\omega)$ be a symplectic vector space and $G \subset \Sp(V)$ a finite
group. We are interested in the singularities of the quotient
$V/G$. In particular, the quotient $V/G$ is said to admit a
(projective) \textit{symplectic resolution} if there exists a
(projective) resolution of singularities $\pi : X \rightarrow V /G$
such that $X$ is a symplectic manifold, see section \ref{sec:defn} for
the precise definition.  

The number of known examples of such symplectic quotient singularities
admitting symplectic resolutions is remarkably small: they are only products
of the following singularities:
\begin{itemize}
\item The 
 infinite series $\C^{2n}/(K \wr S_n)$, where $K$ is a finite subgroup of
$SL_2(\C)$ \\ (here and below, $K \wr S_n := K^n \rtimes S_n$), and
\item Two exceptional quotients $\C^4/G$: the exceptional complex
  reflection group $G_4 < \GL_2(\C) < \Sp_4(\C)$ \cite{Singular,LS},
  and the group $Q_8 \times_{\mathbf{Z}/2} D_8 < \Sp_4(\C)$
  \cite{smoothsra}.
\end{itemize}
We will assume throughout that $V$ is a symplectically irreducible
representation of $G$, i.e., that $V$ does not admit a proper nonzero
symplectic vector subspace invariant under $G$.  As we will recall,
all quotients that admit a symplectic resolution are products of
singularities $V/G$ of this form.

As a consequence of our main result we prove:
\begin{theorem}\label{t:four}
  Let $(V,G)$ be symplectically irreducible. If $V/G$ admits a
  projective symplectic resolution and $\dim V \neq 4$ then $(V,G)
  \simeq (\C^{2n}, K \wr S_n)$ for $K < \SL_2(\C)$, unless possibly
  $(V,G)$ is one of four examples.
\end{theorem}
The four examples referred to above, whose corresponding reflection
representation $V$ has dimension at most 10, will be clarified below,
and in these cases we do not resolve the question of whether $V/G$
admits a projective symplectic resolution.

To state our main theorem, we introduce some more definitions. Recall
that an element $g \in G$ is said to be a \emph{symplectic reflection}
if $\mathrm{rk}( 1 - g) = 2$. The group $G$ (or rather the triple
$(V,\omega,G)$) is said to be a symplectic reflection group if $G$ is
generated by the symplectic reflections that it contains. By
\cite{Verbitsky}, if $V/G$ admits a projective symplectic resolution,
then $G$ is a symplectic reflection group.

If $V$ were not symplectically irreducible, then $V$ would decompose
as $V= V_1 \oplus V_2$ where $V_i$ are symplectic representations of
$G$. When $G$ is a symplectic reflection group as above, then $G$ must
decompose as $G = G_1 \times G_2$ where $G_i < \Sp(V_i)$ is generated
by symplectic reflections in $\Sp(V_i)$; hence $V/G = V_1/G_1 \times
V_2/G_2$.  Therefore the classification of quotients $V/G$ admitting
symplectic resolutions reduces to the case where $V$ is symplectically
irreducible.

A symplectic representation $V$ of $G$ is said to be symplectically
imprimitive if there exists a nontrivial decomposition $V = V_1 \oplus
\cdots \oplus V_k$ into symplectic subspaces such that, for all $i$
and all $g \in G$, there exists $j$ such that $g(V_i) = V_j$.  We call
a group $G < \Sp(V)$ symplectically irreducible or imprimitive if $V$
is such as a representation of $G$.

As above, let $K$ be a finite subgroup of $SL_2(\C)$ (whose
classification is very well known; see \S \ref{sec:Kleinian}
below). The wreath product $K \wr S_n$ acts as a symplectic reflection
group on $\C^{2n}$. By \cite[Theorem 2.2 and 2.9]{CohenQuaternionic},
the symplectically imprimitive and irreducible symplectic reflection
groups are all realized as subgroups (normal when $n > 2$) of $K \wr
S_n$ for suitable $K$ and $n$.

When $n=2$, we will exclude the subgroups $K \rtimes S_2 <
(K \times K) \rtimes S_2$, where $K \hookrightarrow (K \times K)$ is
given by $k \mapsto (k, \alpha(k))$ for some involution $\alpha: K \to
K$.  Our main result reads:
\begin{theorem}\label{thm:main}
  Let $G < \Sp_{2n}(\C)$ be symplectically imprimitive and
  irreducible.  Assume that either $n > 2$ or that $G$ is not of the
  form $K \rtimes S_2$ as above with $K < \SL_2(\C)$. Then the
  symplectic quotient $\C^{2n} / G$
  admits a projective symplectic resolution if and only if $G$ is isomorphic to either $K \wr S_n$
  or $Q_8 \times_{\mathbf{Z}/2} D_8$, which is the group from
  \cite{smoothsra} (for which $n=2$).
\end{theorem}
In more detail, by \cite[Theorem 2.2 and 2.9]{CohenQuaternionic}, the
symplectically imprimitive and irreducible symplectic reflection
groups $G < \Sp_{2n}(\C)$ are, up to conjugation, all of the form:
\begin{itemize}
\item For $n=2$, the group $G=G(K,H,\alpha)$, where $H < K$ is a normal
  subgroup, $\alpha$ is an involution of $K/H$, and $G(K,H,\alpha) <
  K \wr S_2$ is the subgroup generated by $S_2$, $H^2$, and the cosets
  $(kH, \alpha(kH)) < K^2/H^2$ for all $k \in K$;
\item For $n \geq 3$, the group $G=G_n(K,H)$, where $H < K$ is a
  subgroup containing the commutator subgroup $[K,K]$, and $G_n(K,H) <
  K \wr S_n$ is generated by $S_n$, $H^n$, and the cosets $(k_1 H,
  \ldots, k_n H) < K^n/H^n$ for all $k_1, \ldots, k_n \in K$ such that
  $k_1 \cdots k_n \in H$.
\end{itemize}
Using results of Kaledin, we will reduce the theorem to the case
$n = 2$, together with the single case
$G_3(\biDih_2,\cyc_2)$.  The condition that $G(K,H,\alpha)$ is not
the group $K \rtimes S_2$ is precisely saying that $H$ is
nontrivial.  Therefore the main step of the proof is to show that the
groups $G := G(K,H,\alpha)$ do not admit projective symplectic resolutions when
$H \neq \{1\}$.  Let $\ResH$ denote the minimal resolution of $\C^2 /
H$. The key to proving Theorem \ref{thm:main} is to study the action
of the quotient $G / H^2$ on $\ResH \times \ResH$. In particular, we
show that the symplectic variety $(\ResH \times \ResH) / (G / H^2)$
does not, in general, admit a projective symplectic resolution.

\subsection{} The existence of projective symplectic resolutions of the quotient
singularity $V /G$ is known to be equivalent to the existence of a
smooth Poisson deformation of $V/G$; that is, a flat, affine Poisson
deformation of $V/G$ whose generic fibre is a smooth Poisson
variety. Let $\mathsf{H}_{\mathbf{c}}(G)$ be the symplectic reflection
algebra at $t = 0$ associated to $G$ as defined in \cite{EG}. The
centre of this algebra is denoted $\mathsf{Z}_{\mathbf{c}}(G)$. When
the parameter $\mathbf{c}$ is zero, $\mathsf{Z}_{\mathbf{c}}(G)$ is
the coordinate ring of $V/G$. It is known, by \cite{GK}, that
$\mathsf{Z}_{\mathbf{c}}(G)$ defines a flat Poisson deformation of
$V/G$. As noted in \cite[Theorem 1.2.1]{smoothsra}, results of
Ginzburg-Kaledin and Namikawa imply that:

\begin{corollary}
  Let $G$ be a symplectically imprimitive symplectic reflection group obeying the
assumption of Theorem
  \ref{thm:main}.
  Then the variety $\Spec \mathsf{Z}_{\mathbf{c}}(G)$ is singular for
  all parameters $\mathbf{c}$ unless $G$ is isomorphic to $K \wr S_n$ or $Q_8
  \times_{\mathbf{Z}/2} D_8$.
\end{corollary}

Usually one uses the representation theory of symplectic reflection
algebras to show that the variety $\Spec \mathsf{Z}_{\mathbf{c}}(G)$
is singular for all parameters and hence deduce that the corresponding
symplectic quotient singularity does not admit a projective symplectic
resolution. We have taken the opposite approach in this paper.

\subsection{} The paper is structured as follows. In section two we
recall the definition of symplectic variety and symplectic
resolutions. Using work of Namikawa and Kaledin we give two general
criteria for the non-existence of projective symplectic resolutions of $V/G$. In
section three, in order to fix notation, we recall the Kleinian
groups. Cohen's classification of symplectic reflection groups
is recalled in section four.

In section five we consider more specific criteria that can be used to
prove the non-existence of projective symplectic resolutions of $V/G$ when $G$ is
symplectically imprimitive and $V = \C^4$. Then, in section six, we
work through the list of such groups, showing case-by-case that they
do not posses projective symplectic resolutions. In section seven, we deduce the
main result for $\dim V \geq 6$ from these cases and one additional
case, in Lemma \ref{lem:specialthreegroup}. In section eight, we
summarize the resulting proof of Theorem \ref{thm:main}, and deduce
Theorem \ref{t:four} from this.  Finally, in section nine, we
list some open questions (``exercises for the interested reader'').

\subsection{Acknowledgments}

The first author is supported by the EPSRC grant EP-H028153. The
authors would like to thank Y. Namikawa for patiently answering our
questions. The second author was supported by an AIM five-year
fellowship and by NSF grant DMS-0900233.  This material is based upon
work supported by the NSF under Grant No. 0932078 000, while the
authors were in residence at the Mathematical Science Research
Institute (MSRI) in Berkeley, California, during 2013.

\section{Symplectic varieties and symplectic resolutions}

\subsection{} In this section we recall the definition of a symplectic
variety and of symplectic resolutions. We give some criteria for the
(non-)existence of projective symplectic resolutions. The definition of
symplectic variety was introduced by Beauville in the seminal paper
\cite{Beauville}.

\begin{definition}
  Let $X$ be an affine variety. Then $X$ is said to be a
  \textit{symplectic variety} if
\begin{enumerate}
\item $X$ is normal. 
\item There exists a symplectic form $\omega$ on the smooth locus
  $X_{\mathrm{sm}}$ of $X$.
\item There exists a resolution of singularities $\pi : Y \rightarrow
  X$ such that $\pi^* \omega$ extends to a regular $2$-form on $Y$.
\end{enumerate}
\end{definition}

One says that $X$ admits a symplectic resolution if there is a
resolution of singularities $\pi : Y \rightarrow X$ such that $\pi^*
\omega$ extends to a \textit{non-degenerate} $2$-form on $Y$.

\subsection{Conic symplectic varieties} 

Let $X$ be an affine symplectic variety. Then $X$ is said to be
equipped with a good $\Cs$-action if there is an algebraic action of
$\Cs$ on $X$ such that
\begin{enumerate}
\item The weights of $\Cs$ on $X$ are positive and there exists a
  unique fixed point $0 \in X$.
\item The symplectic form $\omega$ has positive weight $l > 0$.
\end{enumerate}

Let $X$ and $Y$ be normal, quasi-projective varieties over $\C$ and
$\Omega_X$, resp. $\Omega_Y$ the corresponding canonical sheaves. A
morphism $\pi : Y \rightarrow X$ is said to be \textit{crepant} if
$\pi^* \Omega_X \simeq \Omega_Y$. If $X$ is a symplectic variety, $Y$
a smooth variety and $\pi : Y \rightarrow X$ a crepant, proper
birational morphism, then $\pi$ is a symplectic resolution. The
composition of two crepant morphisms is again crepant. The following
is a direct consequence of the results of \cite{Namikawa}.

\begin{theorem}\label{thm:abstractexistence}
  Let $X$ be an affine symplectic variety equipped with a good
  $\Cs$-action. Let $\pi : Y \rightarrow X$ be a crepant, projective,
  birational morphism and let $U \subset Y$ be an affine open
  subset. If $X$ admits a projective symplectic resolution then $U$
  admits a projective symplectic resolution.
\end{theorem}

\begin{proof}
  By the proof of \cite[Theorem 5.5]{Namikawa}, since $X$ admits a
  projective symplectic resolution and has a good $\Cs$-action, every
  crepant, projective, birational morphism $\pi_0 : Z_0 \rightarrow X$
  from a space $Z_0$ having at worst $\Q$-factorial terminal
  singularities is necessarily a symplectic resolution. By
  \cite{BCHM}, the minimal model program implies that there exists
  some crepant, projective, birational morphism $\rho : Z \rightarrow
  Y$ such that $Z$ has only $\Q$-factorial terminal
  singularities. Therefore $\pi \circ \rho : Z \rightarrow X$ must be
  a projective symplectic resolution, by the first observation.  In
  particular, $Z$ is a symplectic manifold. The restriction
  $\rho|_{\rho^{-1}(U)} : \rho^{-1}(U) \rightarrow U$ is a resolution
  of singularities. Since $\rho^{-1}(U)$ is a symplectic manifold,
  \cite[Proposition 1.6]{FuSurvey} implies that $\rho|_{\rho^{-1}(U)}$
  is a projective symplectic resolution of $U$.
\end{proof}

\subsection{} The case that will be of interest to us is the
following. Let $(V,\omega,G)$ be a symplectic reflection group acting
on a symplectic vector space $V$ and assume that we are given a normal
subgroup $H$ of $G$ such that
\begin{itemize}
\item $H$ acts on $V$ as a symplectic reflection group.
\item There exists a projective symplectic resolution $\pi : X
  \rightarrow V/H$.
\item The action of $Q := G/H$ lifts to an action on $X$ making $\pi$
  a $G/H$-equivariant morphism.
\end{itemize}
To each $x \in X$, we associate the pair $(Q_x,T_x X)$, where $Q_x$ is
the stabilizer of $x$ in $Q$ and $T_x X$ is the tangent space of $X$
at $x$. Note that $T_x X$ is a symplectic representation of
$Q_x$. This representation is faithful. To see this we note that the
fact that $\pi$ is equivariant implies that $Q$ acts freely on some
dense open subset of $X$. On the other hand, if there is some $g \in
Q_x$ acting trivially on $T_x X$, then this implies that $\dim
\Fix_X(g) = \dim X$ and hence $g = 1$.

\begin{proposition}\label{prop:contradiction}
  If the quotient $V/G$ admits a projective symplectic resolution then
  $T_x X/ Q_x$ admits a projective symplectic resolution for all $x
  \in X$.
\end{proposition}

\begin{proof}
  As in Theorem \ref{thm:abstractexistence}, we may choose some
  crepant projective morphism $\rho : Y \rightarrow X/Q$ such that $Y$
  has at worst $\Q$-factorial, terminal singularities. As shown there,
  the fact that there exists some projective symplectic resolution of
  $V/G$ implies that $Y$ is smooth. Choose $x \in X$ and let $\bar{x}$
  denote its image in $X/Q$. Write $\widehat{Y}_x$ for the completion
  of $Y$ along the closed sub-scheme $\rho^{-1}(\bar{x})$. The
  completion of $X/Q$ at $\bar{x}$ is isomorphic to the quotient of
  the completion $\widehat{T_x X}$ by $Q_x$. Then $\rho$ induces a
  projective morphism $\widehat{\rho} : \widehat{Y}_x \rightarrow
  \widehat{T_x X}/Q_x$. Since $\widehat{Y}_x$ is smooth with trivial
  canonical bundle, $\widehat{\rho}$ is a projective smooth formal
  crepant resolution in the sense of \cite[\S
  1]{KaledinDynkin}. Therefore \cite[Theorem 1.4]{KaledinDynkin}
  implies that $T_x X / Q_x$ admits a projective symplectic
  resolution.
\end{proof}

\section{The Kleinian groups}\label{sec:Kleinian}

\subsection{} Let $\Quat = \R \oplus \R \bi \oplus \R \bj \oplus \R
\bk$ be the skew-field of quaternions. In order to fix notation, we
remark that the finite subgroups of $\GL(\mathbf{H}^1)$ up to
isomorphism are the cyclic group $\cyc_m = \langle \bzeta \ | \
\bzeta^m = 1 \rangle, \ \textrm{of order } m$, the binary dihedral
group $\biDih_m = \langle \cyc_{2m}, \bk \rangle, \ \textrm{of order }
4m$ and the three exceptional groups $\tetra = \langle \biDih_2,
\bomega \rangle$ of order $24$, where $\bomega = \frac{1}{2}(1 - \bi +
\bj + \bk)$, $\octa = \langle \tetra, \frac{1}{\sqrt{2}}(1 + \bi)
\rangle$ of order $48$ and
$$
\isohedral = \langle \biDih_2, \frac{1}{2}(\rho + \sigma \bi - \bj)
\rangle, \ \textrm{of order $120$, where $\rho = 2 \cos \left(
    \frac{\pi}{5} \right)$} \ \textrm{and $\sigma = 2 \cos \left(
    \frac{3 \pi}{5} \right)$.}
$$

\subsection{Complexification}
We consider $\C \subseteq \Quat$ to be the subfield $\C = \R \oplus \R
\cdot \bi$.  Given a finite subgroup $G$ of $\GL(\Quat^n)$, a choice
of complex structure on $\Quat$ realizes $G$ as a subgroup of
$\GL(\C^{2n})$ (we consider matrices acting on the \textit{right} of
$\Quat^n$). The standard choice of complex structure, as used in
\cite{CohenQuaternionic} is $\Quat = \C \oplus \C \bj$. However, in
order to use the results of \cite{ItoNakamuraHilbert}, we choose the
complexification $\Quat = \C \oplus \C \bk$. The procedure $G \subset
\GL(\mathbf{H}^n)$ goes to $G^\vee \subset \GL(\C^{2n})$ is called
\textit{complexification}. Noting that $\epsilon := e^{\frac{2 \pi
    i}{8}} = \frac{1 + i}{\sqrt{2}}$, complexification is uniquely
defined by
$$
\bi = \left(\begin{array}{cc}
i & 0 \\
0 & -i 
\end{array} \right), \quad \bj = \left(\begin{array}{cc}
0 & - i \\
-i & 0 
\end{array} \right), \quad \bk = \left(\begin{array}{cc}
0 & 1 \\
-1 & 0 
\end{array} \right).
$$
The complexification map defines an embedding $\GL(\Quat^n)
\hookrightarrow \Sp(\C^{2n})$, where $\Quat$ is viewed as a
two-dimensional symplectic complex vector space with the form $\langle
1, \bk \rangle = 1$.  As explained in \cite[\S 1]{CohenQuaternionic},
this induces an equivalence between finite subgroups of
$\GL(\Quat^n)$, up to conjugation, and finite subgroups of
$\Sp(\C^{2n})$, up to conjugation.

This explicitly realizes all the finite subgroups of $\GL(\Quat)$
above as subgroups of $\Sp_2(\C) = \SL_2(\C)$.  Moreover, this
explicitly realizes the subgroups of \cite{CohenQuaternionic} as
subgroups of $\Sp_{2n}(\C)$, since they are described there in terms
of quaternions.

\section{Cohen's classification of symplectic reflection groups}
\label{sec:Gdefn}

\subsection{} \label{ss:cohen} The irreducible symplectic reflection
groups were first classified by A. Cohen in
\cite{CohenQuaternionic}. We recall here the the outline of his
classification. (We remark that his results are stated in terms of
quaternionic reflection groups, but based on the results of \cite[\S
1]{CohenQuaternionic}, there is a bijective correspondence between
quaternionic and symplectic reflection groups preserving symplectic imprimitivity
and symplectic irreducibility.)


A symplectic reflection group $G < \Sp(V)$ is said to be
\textit{improper} if it preserves a Lagrangian subspace $L \subseteq
V$, so that $G < \GL(L)$ is actually a complex reflection group.
Complex reflection groups were classified by Chevalley, Shephard, and
Todd, and for these groups our main results are immediate consequences
of \cite{Singular}.  Thus from now on we assume $G$ is proper.


We further say that a symplectic reflection group is \emph{complex
  imprimitive} if it is imprimitive considered as a subgroup of
$\GL(\C^{2n})$, i.e., if there exists a decomposition $\C^{2n} = V_1
\oplus \cdots \oplus V_k$ into complex subspaces such that, for all $g
\in G$ and all $i$, there exists $j$ such that $g(V_i) = V_j$.  There
exist symplectically primitive symplectic reflection groups which are,
nonetheless, complex imprimitive.  Therefore there are three natural
classes to consider (assume $G$ is symplectically irreducible and proper):

\begin{enumerate}
\item The symplectically imprimitive 
  symplectic reflection groups.  These are the groups that we will
  consider in this paper. 
  By \cite[Theorem 2.2]{CohenQuaternionic},
  such subgroups of $\Sp_4(\C)$ are of the form $G(K,H,\alpha)$ 
 as defined in the introduction, and by
  \cite[Theorem 2.6]{CohenQuaternionic}, they must be listed in
  Tables
  \ref{tab:table1} and \ref{tab:table2}. By \cite[Theorem
  2.9]{CohenQuaternionic}, such subgroups of $\Sp_{2n}(\C)$ for $n >
  2$ are of the form $G_n(K,H)$, as defined in the introduction. Conversely,
  all of these groups are symplectically imprimitive (and irreducible)
  symplectic reflection groups.
\item The symplectically primitive symplectic reflection groups which
  are complex imprimitive.  These all lie in $\Sp_4(\C)$ and are
  classified in \cite[Lemma 3.3]{CohenQuaternionic}. The fact that
  those listed are all such groups follows from \cite[Theorem
  3.6]{CohenQuaternionic}.
\item The complex primitive symplectic reflection groups. In this
  case, \cite[Theorem 4.2]{CohenQuaternionic} say that the groups in
  this case are precisely those listed in \cite[Table
  III]{CohenQuaternionic}. Thus there are only thirteen such groups
  and they occur in dimension at most ten.
\end{enumerate}
It is interesting to note that the three exceptional groups
$P_1,P_2$, and $P_3$ that occur in (3) above are all extensions of
some group by the symplectic reflection group $Q_8
\times_{\mathbf{Z}/2} D_8$ studied in \cite{smoothsra}.

\subsection{Imprimitive reflection groups}\label{sec:defn}

We denote by $K$ a finite subgroup of $SL_2(\C)$ and $H$ a normal
subgroup of $K$.  The Kleinian singularity $\C^2 / H$ is denoted $\QuotH$ and
the corresponding minimal resolution is $\pi : \ResH \rightarrow
\QuotH$. 

We let $\alpha$ be an involution of $\Gamma:= K/H$.
We choose coordinates on $\C^2$ so that the ring of polynomial
functions on $\C^2$ is $\C[x,y]$. We also endow $\C^2$ with the
standard symplectic form so that $\{ x, y \} = 1$, where $\{ - , - \}$
is the corresponding Poisson bracket. Associated to $H,K,\alpha$ is
the symplectically imprimitive and irreducible group $G :=
G(K,H,\alpha)$, acting on $\C^4 = \C^2 \times \C^2$, which we defined
in the introduction as a subgroup of $G < K \wr S_2$.  In \S
\ref{sec:Gdefn} below, we recall the action of $G$ and $K \wr S_2$ on
$\C^4$.

Let $Q$ be the quotient $Q := G / H^2 = (K / H) \rtimes_{\alpha} S_2$
of $G$.

If $T$ is a group, $\rho$ a representation of $T$ and $g$ an
automorphism of $T$ then the twist of $\rho$ by $g$ is denoted ${}^g
\rho$. As a vector space, ${}^g \rho = \rho$ and $t \cdot m := g(t) m$
for $t \in T$ and $m \in {}^g \rho$. 

By a variety, we mean a reduced and irreducible scheme of finite type
over $\C$. 

\subsection{} We continue to take $G=G(K,H,\alpha)$. Let $s_{12} \in
S_2 < G < K \wr S_2$ be the transposition. Following \cite[\S
2]{CohenQuaternionic}, let $L_\alpha := \{x \in K \mid x \alpha(x) \in
H\}$.  As observed in \cite{CohenQuaternionic}, the symplectic
reflections in $G$ are the elements of $H \times \{1\}$, $\{1\} \times
H$, and $\{(x,x^{-1}) s_{12} \mid x \in L_\alpha\}$.


\begin{lemma}\label{lem:normalsub}
  The subgroup $H \wr S_2$ is normal in $G$ if and only if $\alpha =
  1$, and the subgroup $H^2$ is always normal in $G$.
\end{lemma}

\begin{proof}
  By considering $k = (x,y) \cdot s_{12}$ acting by conjugation on
  $H^2 \cdot s_{12}$, where $(x,y) \in K^2$ satisfies $y \in \alpha(x
  H)$, we see that $H \wr S_2$ is normal in $G$ if and only if $x H
  y^{-1} = H$ for all $x \in K$ and $y \in \alpha(x H)$. But this is
  the same as saying that $\alpha = 1$ on the quotient. The second
  statement is clear.
\end{proof}

The pair $(\mathrm{id},\alpha)$ defines an embedding of $H / K$ into
$(H/K)^2$. Since $\alpha$ has order at most $2$, $S_2$ preservers the
image of this map and we may form a twisted semi-direct product $(H /
K) \rtimes_{\alpha} S_2$. Then $G / K^2 \simeq (H / K)
\rtimes_{\alpha} S_2$. The involutions in $(H / K) \rtimes_{\alpha}
S_2$ are all given by $\{ (x,\alpha(x)) \ | \ x^2 = 1 \} \cup \{
(x,\alpha(x)) \cdot s_{12} \ | \ x \alpha(x) = 1 \}$ (note that the
latter condition $x \alpha(x) = 1$ can also be written as $x \in L_\alpha/H$).
The group $G / K^2$ is denoted $Q$.


\begin{table}
  \caption{Nonconjugate Symplectically Imprimitive, Symplectically 
Irreducible Proper Four-Dimensional Reflection Groups $G(K,H,\alpha)$}\label{tab:table1}
\centering
\rotatebox{90}{ 
\begin{tabular}{c|llllll} 
  \hline 
  & $K$ & $H$ & $K/H$ & $\alpha$ & $|G(K,H,\alpha)$ & $|L_{\alpha}|$ \\ 
  \hline
  (A) & $\biDih_m$ & $\cyc_{2m}$ & $\cyc_2$ & $1$ & $16m^2$ & $4m$ \\
  (B) & $\biDih_{2ml}$ & $\cyc_{2m}$ & $\Dih_{2l}$ &  $\alpha_r^{\star} \left\{ \begin{array}{ll} 
      \textrm{where} & 0 \le r \le l, r \textrm{ odd } \\
      & l = \gcd (l, (r+1)/2) \gcd (l,(r-1)/2) 
\end{array} \right.$ & $32m^2l$ & $2m \left( \begin{array}{c} 
  \gcd (2l, r-1) \\
  + \\
  \gcd (2l, r+1) \end{array} \right)$ \\
(C) & $\biDih_{(2m +1)l}$ & $\cyc_{2m+1}$ & $\biDih_l$ &  $\beta_r^{\dagger} \left\{ \begin{array}{ll} 
    \textrm{where} & 0 \le r \le l, r \textrm{ odd } \\
    & l = \gcd (l, (r+1)/2) \gcd (l,(r-1)/2) 
\end{array} \right. $ & $8(2m+1)^2l$ & $(2m +1) \left( \begin{array}{c} 
\gcd (2l, r-1) \\
+ \\
\gcd (2l, r+1) \end{array} \right)$ \\
(D) & $\biDih_{2m}$ & $\biDih_m$ & $\cyc_2$ &  $1$ & $64m^2$ & $8m$ \\
(E) & $\biDih_m$ & $\biDih_m$ & $1$ &  $1$ & $32m^2$ & $4m$ \\
(F) & $\biDih_{2m+1}$ & $\cyc_2$ & $\Dih_{2m+1}$ &  $\alpha_r^{\star} \left\{ \begin{array}{ll} 
\textrm{where} & 0 \le r \le m, \\
 & 2m + 1 = \gcd (2m + 1, r +1) \gcd (2m+1,r-1) 
\end{array} \right. $ & $16(2m+1)$ & $2 \left( \begin{array}{c} 
\gcd (2m + 1, r+1) \\
+ \\
\gcd (2m +1, r-1) \end{array} \right)$ \\
(G) & $\biDih_m$ & $1$ & $\biDih_m$ & $\beta_r^{\dagger} \left\{ \begin{array}{ll} 
\textrm{where} & 0 \le r \le m, r \textrm{ odd} \\
 & m = \gcd (m, (r +1)/2) \gcd (m,(r-1)/2) 
\end{array} \right. $ & $8m$ & $\left( \begin{array}{c} 
\gcd (2m, r+1) \\
+ \\
\gcd (2m, r-1) \end{array} \right)$ \\
(H) & $\tetra$ & $\tetra$ & $1$ &  $1$ & $1152$ & $24$ 
\end{tabular}
}
\end{table}

\subsection{}

In Table \ref{tab:table1} the automorphism $\alpha_r \in
\mathrm{Aut}(\Dih_m)$ is defined by $\alpha_r(u) = u^r; \alpha_r(v) =
v$, where $\Dih_m = \langle u,v \ | \ u^m = v^2 = (uv)^2 = 1 \rangle$
and the automorphism $\beta_r \in \mathrm{Aut}(\biDih_m)$ is defined
by $\beta_r(\bzeta_m) = \bzeta_m^r$ and $\beta_r(\bk) = - \bk$.

\begin{table}
\caption{Nonconjugate Symplectically
Imprimitive, Symplectically Irreducible Proper Four-Dimensional Reflection Groups $G(K,H,\alpha)$ (cont.)}\label{tab:table2}
\centering 
\begin{tabular}{c|llllll} 
\hline 
 & $K$ & $H$ & $K/H$ & $\alpha$ & $|G(K,H,\alpha)$ & $|L_{\alpha}|$ \\ 
\hline
(I) & $\tetra$ & $\biDih_2$ & $\cyc_3$ & Inversion & $384$ & $24$ \\
(J) & $\tetra$ & $\cyc_2$ & $\mathrm{Alt}(4)$ & Conjugation by $(12)$ & $96$ & $12$ \\
(K) & $\tetra$ & $1$ & $\tetra$ & Conjugation by $(i - j)$ & $48$ & $12$ \\
(L) & $\octa$ & $\octa$ & $1$ & $1$ & $4608$ & $48$ \\
(M) & $\octa$ & $\tetra$ & $\cyc_2$ & $1$ & $2304$ & $48$ \\
(N) & $\octa$ & $\biDih_2$ & $\Dih_3$ & $1$ & $768$ & $32$ \\
(O) & $\octa$ & $\cyc_2$ & $\mathrm{Sym}(4)$ & $1$ & $192$ & $14$ \\
(P) & $\octa$ & $1$ & $\octa$ & Conjugation by $k$ & $96$ & $18$ \\
(Q) & $\octa$ & $1$ & $\octa$ & $1 \neq \alpha \in \mathrm{Ker} : \mathrm{Aut}(\octa) \rightarrow \mathrm{Aut}(\mathrm{Sym}(4))$ & $96$ & $14$ \\
(R) & $\isohedral$ & $\isohedral$ & $1$ & $1$ & $14400$ & $120$ \\
(S) & $\isohedral$ & $\cyc_2$ & $\mathrm{Alt}(5)$ & $1$ & $480$ & $32$ \\
(T) & $\isohedral$ & $\cyc_2$ & $\mathrm{Alt}(5)$ & Conjugation by $(12)$ & $480$ & $20$ \\
(U) & $\isohedral$ & $1$ & $\isohedral$ & Conjugation by $j$ & $240$ & $30$ \\
(V) & $\isohedral$ & $1$ & $\isohedral$ & Preimage of conjugation by $(12)$ & $240$ & $20$ \\
 & & & & under $\mathrm{Aut}(\isohedral) \rightarrow \mathrm{Aut}(\mathrm{Alt}(5))$ & & \\
\hline
\end{tabular}
\end{table}

\section{Singular subgroups of $G$}

This section is rather technical, therefore we provide an outline. We
wish to show that the $G(K,H,\alpha)$ with $H \neq \{1\}$ 
do not admit
symplectic resolutions. The purpose of this section is to provide two
general criteria, Theorem \ref{thm:singularimpliesnores} and
Proposition \ref{prop:singc2}, for the non-existence of projective symplectic
reflections for these groups. This is done by analyzing the set of
points on $\ResH \times \ResH$ that have non-trivial stabilizer under
the action of the group $Q$ and showing that this set has components
of dimension zero.

\subsection{} A subgroup $P$ of $\Gamma$ is said to be a
\textit{parabolic} subgroup if there exists some $x \in \ResH$ such
that $\Stab_{\Gamma}(x) = P$. The set of all points $x$ in $\ResH$
such that $\Stab_{\Gamma}(x) = P$ is denoted $\ResH (P)$. It is a
smooth, though not generally connected, locally closed subset of
$\ResH$. The fact that $\ResH$ is smooth implies that each connected
component of $\ResH (P)$ is pure-dimensional. We denote by $\ResH
(P)_0$, resp. $\ResH (P)_{> 0}$, the union of all components of
dimension zero, resp. greater than zero, in $\ResH (P)$. Key to our
classification theorem is the following technical definition.

\begin{definition}\label{defn:singular}
Let $P$ be a parabolic subgroup of $\Gamma$.  
\begin{enumerate}
\item A point $x \in \ResH (P)_0$ is \textit{isolated} if there exists
  no parabolic subgroup $\{1 \} \neq T \subsetneq P$ such that $x \in
  \overline{\ResH (T)}$.
\item The group $P$ is said to be \textit{singular} if there exist
  isolated points $x \in \ResH (P)_0$ and $y \in \ResH (\alpha(P))_0$
  such that $y \notin \Gamma \cdot x$.
\end{enumerate}
\end{definition}

Note that a singular subgroup of $\Gamma$ is actually the data of a
four-tuple $(P,x;\alpha(P),y)$.

\begin{lemma}\label{lem:niceopen}
  If $P$ is a singular subgroup of $\Gamma$ then there exists an
  affine open, $Q$-stable subset $U \subset \ResH \times \ResH$ and
  closed point $p \in U$ such that $\Stab_{Q}(p) = P$ and $Q$ acts
  freely on $U \backslash Q \cdot p$, where $P$ is realized as a
  subgroup of $Q$ via $(\mathrm{id},\alpha)$.
\end{lemma}

\begin{proof}
  Let $P = P_1, P_2 = \alpha(P), \ds, P_k$ be all the conjugates of
  either $P$ or $\alpha(P)$ in $\Gamma$. The dense open subset of
  $\ResH \times \ResH$ consisting of all points with trivial
  stabilizer is denoted $(\ResH \times \ResH)_{\reg}$ and write
  $\ResH(P_i)_{\mathrm{iso}}$ for the set of all isolated points in
  $\ResH(P_i)_0$. Set
$$
U_0 = (\ResH \times \ResH)_{\reg} \cup \left[ \bigcup_{i = 1}^k
  \ResH(P_i)_{\mathrm{iso}} \times
  \ResH(\alpha(P_i))_{\mathrm{iso}}\right].
$$
It is a $Q$-stable, dense subset of $\ResH \times \ResH$. Let 
$$
\Delta = \bigcup_{h \in \Gamma} \{ (h \cdot x, x) \in \ResH \times
\ResH \ | \ x \in \ResH \}.
$$
It is a proper, closed, $Q$-stable subvariety of $\ResH \times
\ResH$. We set $U_1 = U_0 \backslash \Delta$ and claim that $U_1$ is a
$Q$-stable, dense \textit{open} subset of $\ResH \times \ResH$. That
it is $Q$-stable and dense is straight-forward. To see that it is
open, we note firstly that the stabilizer of any $x \in U_1$ is either
trivial or $(\mathrm{id} \times \alpha)(P_i)$. This basically follows
from the fact that if $s_{12} \cdot (h,\alpha(h))$ stabilizes $x$ for
some $h \in \Gamma$, then $x \in \Delta$. Now decompose
$$
(\ResH \times \ResH) \backslash U_1 = \bigsqcup_{\beta} C_{\beta}
$$
into the connected components of the stabilizer stratification of
$(\ResH \times \ResH) \backslash U_1$. To show that $U_1$ is open, it
suffices to show that $\overline{C}_{\beta} \cap U_1 = \emptyset$:
clearly $\overline{C}_{\beta} \cap (\ResH \times \ResH)_{\reg} =
\emptyset$, therefore if $\overline{C}_{\beta} \cap U_1 \neq
\emptyset$, then there exists some point $x \in \overline{C}_{\beta}
\backslash C_{\beta} \cap U_1$ whose stabilizer is $(\mathrm{id}
\times \alpha)(P_i)$. If $y \in C_{\beta}$ then $\Stab_{Q}(y)
\subseteq (\mathrm{id} \times \alpha)(P_i)$, which implies that $y =
(y_1,y_2)$ with $\Stab_{\Gamma}(y_1) \subseteq P_i$. In this case
there exists some connected components $D_1$ of
$\ResH(\Stab_{\Gamma}(y_1))$ and$D_2$ of
$\ResH(\alpha(\Stab_{\Gamma}(y_1)))$ such that $C_{\beta} \subset D_1
\times D_2$ is a dense subset.  This implies that $x \in \overline{D_1
  \times D_2}$, contradicting the fact that $x = (x_0,x_1)$ with each
$x_i$ isolated.

\subsection{} By assumption, there exists some $x \in \ResH(P)_{0}$
and $y \in \ResH(\alpha(P))_0$ such that $y \notin \Gamma \cdot
x$. Set $p = (x,y)$. The condition $y \notin \Gamma \cdot x$ implies
that $p \notin \Delta$. Therefore it belongs to $U_1$. We may choose
an open subset $U_2$ of $p$ in $U_1$ such that the stabilizer of every
point in $U_2$ (except $p$ itself) has trivial stabilizer. Then the
set $U$ we require is $\bigcup_{g \in Q} g \cdot U_2$, except that $U$
may not be affine. However, \cite[Lemma 1.3]{BertinNotes} shows that
we can replace $U$ by a smaller affine open subset.
\end{proof}

\begin{theorem}\label{thm:singularimpliesnores}
  If there exists a singular subgroup $P$ of $\Gamma$ then $V / G$
  does not admit a projective symplectic resolution.
\end{theorem}

\begin{proof}
Let 
$$
F = \bigcup_{1 \neq h \in Q} \mathrm{Fix}(h)
$$
be the closed subset of $\ResH \times \ResH$ consisting of all points
with non-trivial stabilizer. If $U$ is the open subset of $\ResH
\times \ResH$ whose existence is guaranteed by Lemma
\ref{lem:niceopen}, then $U \cap F = Q \cdot p$ is a finite union of
points. Let $x \in U \cap F$ and $Q_x$ the stabilizer of $x$. By
assumption $Q_x \neq 0$. By assumption, $Q_x$ acts freely on $T_x U
\backslash \{ 0 \}$, therefore $T_x U / Q_x$ is an isolated symplectic
singularity. Hence $Q_x$ contains no symplectic reflections. Therefore
$(T_x U,Q_x)$ does not admit any symplectic resolution by Verbitsky's
Theorem, \cite{Verbitsky}. Then Proposition \ref{prop:contradiction}
implies that $V / G$ does not admit a projective symplectic
resolution.
\end{proof}

\subsection{}\label{sec:singularconidtion} The case where $H = \cyc_2$
occurs several times in Tables \ref{tab:table1} and
\ref{tab:table2}. In this case $\ResH = T^* \mathbf{P}^1$ and the
action of $\Gamma$ on $T^* \mathbf{P}^1$ comes from the embedding
$\Gamma \hookrightarrow PSL_2(\C) = \mathrm{Aut}(\mathbf{P}^1)$. Let
$V = \C^2$ with basis $v_1,v_2$ and $x_1, x_2$ the dual basis of $V^*$
so that $\C[V] = \C[x_1,x_2], \C[V^*] = \C[v_1,v_2]$. Then
$$
T^* \mathbf{P}^1 = \{ (v,w) \in (V \backslash \{ 0 \}) \times V^* \ |
\ w(v) = 0 \} / \C^{\times}.
$$
The charts $U_1 = (x_2 \neq 0)$, $U_2 = (x_1 \neq 0)$ cover $T^*
\mathbf{P}^1$ and
$$
\C[U_1] = \C \left[ \frac{x_1}{x_2}, v_1 x_2 \right], \quad \C[U_2] =
\C \left[ \frac{x_2}{x_1}, v_2 x_1 \right].
$$
We say that $T \subset \Gamma$ is a \textit{maximal cyclic} subgroup
of $\Gamma$ if it is a cyclic subgroup of $\Gamma$ and there is no
other cyclic subgroup of $\Gamma$ that properly contains $T$.

\begin{lemma}\label{lem:maxcyclic}
  Let $p \in \mathbf{P}^1$ and $T = \Stab_{\Gamma}(p)$. Then $T$ is a
  maximal cyclic subgroup of $\Gamma$.
\end{lemma}

\begin{proof} 
  Assume that $1 \neq g \in \Gamma$ fixes $p$. Diagonalizing $g$, we
  may assume that $p = [1:0]$ or $[0:1]$. This implies that every
  non-zero $g$ fixes exactly two points in $\mathbf{P}^1$. If $h$ is
  another non-zero element of $\Gamma$ that fixes $[1:0]$ say then
$$
g = \left(\begin{array}{cc}
    \eta & 0 \\
    0 & \eta^{-1}
\end{array} \right), \quad h = \left(\begin{array}{cc}
\zeta & b \\
0 & \zeta^{-1} 
\end{array} \right), 
$$
where $\eta, \zeta$ are some roots of unity. The matrix $h$ above has
infinite order if $b \neq 0$. Therefore the fact that $\Gamma$ is
finite implies that $g$ and $h$ are contained in some common cyclic
subgroup of $\Gamma$ and fix the same points of $\mathbf{P}^1$. Hence
the stabilizer of $[1:0]$ in $\mathbf{P}^1$ is some maximal cyclic
subgroup of $\Gamma$.
\end{proof}

\begin{proposition}\label{prop:singc2}
  If there exists a point $p \in \mathbf{P}^1$ such that
  $|\Stab_{\Gamma}(p)| > 2$, then there exists a point $x \in T^*
  \mathbf{P}^1 \times T^* \mathbf{P}^1$ such that $(Q_x, T_x (T^*
  \mathbf{P}^1 \times T^* \mathbf{P}^1))$ does not admit a 
  symplectic
  resolution.
\end{proposition}

\begin{proof}
  By Lemma \ref{lem:maxcyclic}, the stabilizer of $p$ is a maximal
  cyclic subgroup of $\Gamma$. Therefore we may assume that
$$
\Stab_{\Gamma}(p) = \langle s \rangle \simeq C_{m}
$$
for some $m > 2$. Let $r = \alpha(s)$. Thinking of $s$ and $r$ as
elements of order $2m$ in $K$, we may assume that $s$ is a diagonal
matrix with respect to the basis $v_1, v_2$ of $V$. We choose another
basis $w_1,w_2$ of $V$ and dual basis $y_1, y_2$ of $V^*$ such that
$r$ is diagonal with respect to these basis. Then there exists a
primitive $2m^{th}$ root of unity $\zeta$ such that
$$
s \cdot v_1 = \zeta v_1, \quad s \cdot v_2 = \zeta^{-1} v_2, \quad r
\cdot w_1 = \zeta^a w_1, \quad r \cdot w_2 = \zeta^{-a} w_2,
$$
for some integer $a$ coprime to $2m$. Let $p_1 = [1:0], p_2 = [0:1]$,
resp. $q_1 = [1:0], q_2 = [0:1]$, with respect to the coordinates
$[x_1:x_2]$, resp. $[y_1: y_2]$ of $\mathbf{P}^1$. Then $s \cdot p_i =
p_i$ and $r \cdot q_i = q_i$ for $i = 1,2$. Let $x = (p_1,q_1) \in T^*
\mathbf{P}^1 \times T^* \mathbf{P}^1$. Recall that $Q = (H/K)
\rtimes_{\alpha} S_2$. It is a subgroup of $(H / K) \wr S_2 = (H/K)^2
\cup (H/K)^2 \cdot s_{12}$. We decompose $T = \Stab_{Q}(x)$ as
$$
T_0 = T \cap (H/K)^2, \quad T_1 = (H/K)^2 \cdot s_{12}. 
$$
There are two cases to consider: case $a)$ $T_1 = \emptyset$, and case
$b)$ $T_1 \neq \emptyset$.

We begin by considering case $a)$. In this case we have $T = T_0 =
\langle s_{(1)} r_{(2)} \rangle$. Let $e_1,e_2,f_1,f_2 =
(\frac{x_1}{x_2})^*, (v_1 x_2)^*, (\frac{y_1}{y_2})^*, (w_1 y_2)^*$ so
that
$$
T_x(T^* \mathbf{P}^1 \times T^* \mathbf{P}^1)) = \C \cdot \{
e_1,e_2,f_1,f_2 \}.
$$ 
With respect to this chosen basis of $T_x(T^* \mathbf{P}^1 \times T^*
\mathbf{P}^1)$, we have
$$
s_{(1)} r_{(2)} = \left( \begin{array}{cccc}
    \zeta^2 & 0 & 0 & 0 \\
    0 & \zeta^{-2} & 0 & 0 \\
    0 & 0 & \zeta^{2a} & 0 \\
    0 & 0 & 0 & \zeta^{-2a}
\end{array} \right). 
$$
This is not a symplectic reflection. Therefore $(Q_x,T_x(T^*
\mathbf{P}^1 \times T^* \mathbf{P}^1))$ does not admit a symplectic
resolution by Verbitsky's Theorem, \cite{Verbitsky}.

In the second case, there exists some $h \in \Gamma$ such that $t =
s_{12} h_{(1)} \alpha(h)_{(2)} \in T_1$ (in fact, $T_1 = T_0 \cdot
t$). Hence $h \cdot p_1 = q_1$ and $\alpha(h) \cdot q_1 = p_1$. This
implies that
$$
\alpha(h) h = s^{\lambda}, \quad h \alpha(h) = r^{\mu},
$$
for some $\lambda$ and $\mu$. Applying $\alpha$ to the above equations
shows that in fact $\lambda = \mu$. Moreover
$$
h^{-1} \Stab_{\Gamma}(q_1) h = \alpha(h) \Stab_{\Gamma}(q_1) h = \Stab_{\Gamma}(p_1).
$$
Therefore, $h \cdot p_2 = q_2$ and $\alpha(h) \cdot q_2 = p_2$. If
$\Stab_Q(x)$ is generated by symplectic reflections then we may assume
that $t$ is a symplectic reflection. We have
$$
t^2 = \alpha(h)_{(1)} h_{(1)} h_{(2)} \alpha(h)_{(2)} = (s_{(1)} r_{(2)})^{\lambda}. 
$$
This cannot be a symplectic reflection. Therefore it must be the
identity, $\lambda = 0$, and $t^2 = \pm \mathrm{id}$ in $G$. Possibly
after rescaling, we have $h \cdot x_1 = y_1$ and $h \cdot x_2 =
y_2$. Thus, with respect to the basis $e_1,e_2,f_1,f_2$ above, $t$ is
given by
$$
t = \left( \begin{array}{cccc}
    0 & 0 & 1 & 0 \\
    0 & 0 & 0 & 1 \\
    1 & 0 & 0 & 0 \\
    0 & 1 & 0 & 0
\end{array} \right). 
$$
Then $\Stab_Q(x) = C_m \cup C_m \cdot t$. The characteristic
polynomial of $(s_{(1)} r_{(2)})^i \cdot t$ is $u^4 - (\zeta^{2i(a+1)}
+ \zeta^{-2i(a+1)}) u^2 + 1$. Therefore it is a symplectic reflection
if and only if $\zeta^{2i(a+1)} + \zeta^{-2i(a+1)} = 2$ i.e. $i(a+1) =
0$ modulo $m$. If $a + 1$ is not zero modulo $m$ then the solutions to
this equation form a proper subgroup $C_d$ of $C_m$. Then the subgroup
of $Q$ generated by all symplectic reflections is the proper subgroup
$G(d,d,2)$ of $Q$. Thus, for $\Stab_Q(x)$ to be a symplectic
reflection group we must have $a = -1$. In this case,
$$
(\Stab_Q(x), T_x (T^* \mathbf{P}^1 \times T^* \mathbf{P}^1)) \simeq
(G(m,m,2),\mathfrak{h} \oplus \mathfrak{h}^*),
$$
where $\mathfrak{h}$ is the reflection representation for
$G(m,m,2)$. By \cite{Singular}, such a pair admits a symplectic
resolution only if $m = 1,2$.
\end{proof}

\begin{remark}
  Note that there will exist a point $p \in \mathbf{P}^1$ such that
  $|\Stab_{\Gamma}(p)| > 2$ if and only if there is an element in $K$
  of order at least $6$.
\end{remark}

\section{The action of $\Gamma$ on $\Hilb^H(\C^2)$}

In this section we show case-by-case (following Tables
\ref{tab:table1} and \ref{tab:table2}) that the symplectically
imprimitive and irreducible symplectic reflection groups
$G(K,H,\alpha)$ with $H \neq \{1\}$
 satisfy at least one of the criterion given by
Theorem \ref{thm:singularimpliesnores} and Proposition
\ref{prop:singc2}. Our results are summarized in Theorem
\ref{thm:rankfourclass}.

\subsection{} The minimal resolution of $\C^2 / H$ is denoted
$\ResH$. We denote by $\Hilb^n \C^2$ the Hilbert scheme of $n$ points
in the plane. This is a smooth, symplectic manifold of dimension $2n$,
see \cite{NakajimaBook}.

\begin{proposition}
  The action of $G / H^2$ on $\QuotH \times \QuotH$ lifts to an action
  of $G / H^2$ on $\ResH \times \ResH$.
\end{proposition} 

\begin{proof}
  Let $n = |H|$. The action of $K$ on $\C^2$ induces an action of $K$
  on $\Hilb^n \C^2$. One can realize $\ResH$ as the component $\Hilb^H
  \C^2$ of $(\Hilb^n \C^2)^H$ whose generic point $I$ is a radical
  ideal (or in other words $V(I)$ is a free $H$-orbit). This is a
  $K$-stable subvariety of $(\Hilb^n \C^2)^H$ and $\ResH \times \ResH$
  is a $G$-stable subvariety of $\Hilb^n \C^2 \times \Hilb^n \C^2$. By
  definition, the action of $G$ on $\ResH \times \ResH$ factors
  through $G / H^2$.
\end{proof}

\subsection{} We will identify $\ResH$ with $\Hilb^H \C^2$ throughout
this section. By the McKay correspondence, the vertices $\{ \rho_i
\}_{i \in I}$ of the Dynkin diagram can be put in bijection with the
irreducible components of the exceptional fiber of the minimal
resolution of $\C^2 / H$ in such a way that the edges between two
vertices are in bijection with the number of points (which is always
$0$ or $1$) of intersection of the two irreducible components. The
component labeled by the vertex $\rho_i$ is denoted $\mathbf{P}_i$. On
the other hand, the vertices of the \textit{affine} Dynkin diagram can
be naturally labeled by the isomorphism classes $\Irr (H)$ of simple
$H$-modules such that $\dim \Hom_{H}(\C^2 \o \rho_i,\rho_j)$ is twice
the number of edges between the vertices $\rho_i$ and $\rho_j$ (note
that the representation $\C^2$ is self-dual so that $\dim
\Hom_{H}(\C^2 \o \rho_i,\rho_j) = \dim \Hom_{H}(\C^2 \o
\rho_j,\rho_i)$. The trivial representation $\rho_0$ labels an
extending vertex of the affine diagram. Therefore $\Irr_* (H) = \Irr
(H) \backslash \{ \mathrm{triv} \}$ labels the vertices of the
non-affine Dynkin diagram. This allows us to define two actions of
$\Gamma$ on the Dynkin diagram of $H$ by graph automorphisms. The
\textit{geometric action} is defined by $g \cdot \rho_i = \rho_j$ if
$g(\mathbf{P}_i) = \mathbf{P}_j$ and the edge labeled by $p \in
\mathbf{P}_{i_1} \cap \mathbf{P}_{i_2}$ is sent to the edge labeled by
$g \cdot p \in g(\mathbf{P}_{i_1}) \cap g(\mathbf{P}_{i_2})$. The
\textit{representation action} is defined by $g \cdot \rho_i = \rho_j$
if ${}^g \rho_i \simeq \rho_j$ and if $g \cdot \rho_{i_k} =
\rho_{j_k}$, for $k = 1,2$ with $\dim \Hom_{H}(\C^2 \o
\rho_{i_1},\rho_{i_2}) = 2$ then the fact that $g \in SL_2(\C)$
normalizes $H$ implies that ${}^g \C^2 \simeq \C^2$, hence
$$
\dim \Hom_{H}(\C^2 \o \rho_{i_1},\rho_{i_2}) = \dim \Hom_{H}(\C^2 \o
\rho_{j_1},\rho_{j_2})
$$
and so $g$ take the edge between $\rho_{i_1}$ and $\rho_{i_2}$ to the
edge between $\rho_{j_1}$ and $\rho_{j_2}$. Using $\Hilb^H \C^2$, Ito
and Nakumura, \cite{ItoNakamuraHilbert}, constructed a natural
bijection between the irreducible components of the exceptional fiber
and $\Irr_* (H)$ in such a way that the geometric action and the
representation action become equal (a beautiful case free proof was
later given by Crawley-Boevey \cite{CBKleinian}). We recall their
bijection. They showed that for $I \in \pi^{-1}(0)$, the socle of
$\C[x,y]/I$ is either irreducible as a $H$-module or consists of a
pair of non-isomorphic simple $H$-modules. Moreover, if $I \in
\mathbf{P}_i$ does not belong to any other component then the socle of
$\C[x,y]/I$ is irreducible and the isomorphism class of this simple
module depends only on $\mathbf{P}_i$ (and not on the specific choice
of $I$). Hence we may label the Dynkin diagram so that the socle of
$\C[x,y]/I$, with $I \in \mathbf{P}_i$ generic, is $\rho_i$. If $I \in
\mathbf{P}_i \cap \mathbf{P}_j$ then they showed that
$$
\mathrm{soc} (\C[x,y] / I) \simeq \rho_i \oplus \rho_j, \quad \textrm{
  and } \quad \dim \Hom_{H}(\C^2 \o \rho_{i},\rho_{j}) = 2.
$$
If the socle of $\C[x,y] / I$ is isomorphic to $\lambda$ say as a
$H$-module then

\begin{lemma}
  The bijection $\rho_i \leftrightarrow \mathbf{P}_i$ intertwines the
  geometric action and the representation action of $\Gamma$ on the
  Dynkin diagram of $H$.
\end{lemma}

\begin{proof}
  It is straight-forward, but we include a brief explanation for the
  readers convenience. Let $g \in K$. Applying $g$ to the short exact
  sequence $0 \rightarrow I \rightarrow \C[x,y] \rightarrow \C[x,y] /
  I \rightarrow 0$ and using the fact that ${}^g \C[x,y] \simeq
  \C[x,y]$ as a $H$-module implies that ${}^g( \C[x,y] / I) \simeq
  \C[x,y] / g (I)$. The identification restricts to ${}^g
  \mathrm{soc}( \C[x,y] / I) \simeq \mathrm{soc}( \C[x,y] / g (I))$.
\end{proof} 

The above lemma gives us an easy, representation theoretic way to
describe the action of $\Gamma$ on the irreducible components of the
exceptional fiber.

\begin{remark}
  The action of $\Gamma$ on $\ResH$ preserves the symplectic
  form. Therefore each component of the closed subvariety
  $(\ResH)^{\Gamma}$ is either two-dimensional or
  zero-dimensional. Since $\ResH$ is irreducible and the action of
  $\Gamma$ is effective, we see that $(\ResH)^{\Gamma}$ is actually a
  finite collection of points.
\end{remark}

\subsection{(A)}

In this case we have $K = \cyc_{2m}$, $H = \biDih_{m}$ and $\Gamma =
\cyc_{2}$, which is generated by the image $g$ of $\bk$. The
irreducible representations of $\cyc_{2m}$ are labeled $\rho_{i}$, $0
\le i \le 2m - 1$ and ${}^g \rho_{i} = \rho_{2m - i}$. Therefore the
only irreducible representations fixed by $\cyc_2$ are $\rho_0$ and
$\rho_m$. As described above, the irreducible component of the
exceptional locus corresponding to $\rho_i$ is denoted $\mathbf{P}_i$
so that $\mathbf{P}_m$ is the only component that is mapped to itself
by $g$. There are exactly two points $p,q$ in $\mathbf{P}_m$ that are
fixed $g$. Obviously, $q \notin \Gamma \cdot p$. Since $\alpha = 1$,
$(\cyc_{2},p;\cyc_2,q)$ is a singular subgroup of $\cyc_2$.


\subsection{(B)}


In this case we have $H = \cyc_{2m}$, $K = \biDih_{2ml}$ and $\Gamma =
\Dih_{2l} = \langle u,v \rangle$, where $u$ is the image of $\bzeta$
and $v$ is the image of $\bk$. The element $u$ acts trivially on $\Irr
(\cyc_{2m})$ and the action of $v$ is given by ${}^v \rho_{i} =
\rho_{2m - i}$. Therefore $\Stab_{\Dih_{2l}}(\rho_i) = \cyc_{2l}$ if
$1 \le i \neq m \le 2m-1$ and $\Stab_{\Dih_{2l}}(\rho_m) =
\Dih_{2l}$. \textit{Provided} $m \neq 1$, we choose some $1 \le i \neq
m \le 2m-1$ and let $p,q \in \mathbf{P}_i$ be the two points whose
stabilizer is $\cyc_{2l} = \langle u \rangle$. Since $v$ maps
$\mathbf{P}_i$ to $\mathbf{P}_{2m - i}$, we have $q \notin \Gamma
\cdot p$. Noting that $\alpha_r (\cyc_{2l}) = \cyc_{2l}$,
$(\cyc_{2l},p;\cyc_{2l},q)$ is a singular subgroup.  \\

Assume now that $m = 1$. Then we are in the situation described in
(\ref{sec:singularconidtion}). When $l = 1$, the group $G$ is the
subject of the paper \cite{smoothsra}, where it is shown that the
corresponding quotient singularity admits a projective symplectic
resolution. The group $\biDih_{2l}$ contains a cyclic subgroup of
order $4l$. Therefore, when $l \ge 2$, there is an element in $K$ of
order $\ge 8$. Therefore we may apply Proposition \ref{prop:singc2}.


\subsection{(C)}

In this case we have $H = \cyc_{2m+1}$, $K = \biDih_{(2m+1)l}$ and
$\Gamma = \biDih_{l} = \langle g,t \rangle$, where $g$ is the image of
$\bzeta$ and $t$ is the image of $\bk$. The element $g$ acts trivially
on $\Irr (\cyc_{2m+1})$ and ${}^t \rho_{i} = \rho_{2m + 1 -
  i}$. Therefore $\Stab_{\biDih_{l}}(\rho_i) = \cyc_{2l}$ for all $1
\le i \le 2m$. Choose some $1 \le i \le 2m$ and let $p,q \in
\mathbf{P}_i$ be the two points whose stabilizer is $\cyc_{2l}$. Then
$(\cyc_{2l},p;\cyc_{2l},q)$ is a singular subgroup for all $1 \le i
\le 2m$.


\subsection{(D)}

In this case we have $H = \biDih_{m}$, $K = \biDih_{2m}$ and $\Gamma =
\cyc_{2} = \langle g \rangle$, where $g$ is the image of
$\bzeta$. Note that $g(\bzeta^2) = \bzeta^2$ and $g(\bk) = \bzeta^2
\bk$. The element $g$ acts trivially on all irreducible
representations of $H$ except for two of the non-trivial
one-dimensional representations, which are swapped. Take any
$\mathbf{P}^1$ that is fixed by $\Gamma$ and let $p,q \in
\mathbf{P}^1$ be the two points whose stabilizer is $\Gamma$. Then
$(\cyc_2,p;\cyc_2,q)$ is a singular subgroup of $\Gamma$.

\subsection{(I)}

In this case we have $H = \biDih_{2}$, $K = \tetra$ and $\Gamma =
\cyc_3 = \langle g \rangle$, where $g$ is the image of $\bomega$. Then
$g$ permutes cyclically the three non-trivial one dimensional
representations of $H$ and fixes the unique irreducible two
dimensional representation, in addition to fixing the trivial
representation. Let $\mathbf{P}_2$ be the projective line labeled by
the two-dimensional irreducible representation of $H$ and $p,q \in
\mathbf{P}_2$ the two points whose stabilizer is $\Gamma$. Then
$(\cyc_3,p;\cyc_3,q)$ is a singular subgroup of $\Gamma$.




\subsection{(M)}

In this case we have $H = \tetra$, $K = \octa$ and $\Gamma = \cyc_2 =
\langle g \rangle$, where $g$ is the image of $\frac{1}{\sqrt{2}}(1 +
\bi)$. The element $g$ swaps the two non-trivial one dimensional
irreducible representations, swaps the two irreducible two dimensional
representation that are not isomorphic to the realization of $\tetra$
in $\SL_2(\C)$ and fixes all other irreducible representations
(i.e. it is the obvious symmetry of the Dynkin diagram coming from
taking duals of representations). Therefore if $\mathbf{P}$ is one of
the two exceptional components that is labeled by a non-trivial,
self-dual irreducible representation of $\tetra$ then there are
exactly two points $p,q$ in $\mathbf{P}$ whose stabilizer is
$\cyc_2$. Then $(\cyc_2,p;\cyc_2,q)$ is a singular subgroup of
$\Gamma$.

 
\subsection{(N)}

In this case we have $H = \biDih_{2}$, $K = \octa$ and $\Gamma =
\Dih_{3} = \langle u,v \rangle$, where $u$ is the image of $\bomega$
and $v$ the image of $\frac{1}{\sqrt{2}}(1 + \bi)$ in $\Gamma$. Hence
$u^3 = v^2 = 1$. We label the irreducible representations of
$\biDih_2$ so that $\rho_0$ is the trivial representation, $\rho_2$ is
the two dimensional representation and $\rho_1,\rho_3$ and $\rho_4$
are the three non-trivial one-dimensional representations. Then
$$
u \cdot \rho_1 = \rho_4, \quad u \cdot \rho_4 = \rho_3, \quad u \cdot
\rho_3 = \rho_1,
$$
and $u$ fixes all other representations. Similarly, $v$ swaps $\rho_3$
and $\rho_4$ and fixes all other representations. Thus,
$\Stab_{\Gamma}(\rho_1) = \cyc_2$. The stabilizer of $\mathbf{P}_1$ is
$\cyc_2 = \langle v \rangle$. Let $p_1,q_1 \in \mathbf{P}_1$ be the
two points whose stabilizer is $\cyc_2$. The group $\alpha(\cyc_2)$
will also fix one of the extremal vertices $\rho_1$, $\rho_3$ or
$\rho_4$, without loss of generality we assume that it is
$\rho_3$. The points in $\mathbf{P}_3$ whose stabilizer is
$\alpha(\cyc_2)$ are $p_3,q_3$ say. Since both $\cyc_2$ and
$\alpha(\cyc_2)$ fix the central vertex $\rho_2$, one of the two
points $p_i$ or $q_i$ must be the intersection point $\mathbf{P}_2
\cap \mathbf{P}_i$. Let's say its $p_i$ in both cases. Then there can
be no element of $\Dih_{3}$ that maps $p_1$ to $q_3$. Thus,
$(\cyc_2,p_1;\alpha(\cyc_2),q_3)$ is a singular subgroup of $\Gamma$.


\subsection{The cases where $H = \cyc_2$}

In the cases (F),(J),(O), (S), and (T) we have $H =
\cyc_{2}$. Therefore we are in the situation described in
(\ref{sec:singularconidtion}). In all these cases the group $K$
contains at least one element of order $\ge 6$. Therefore we may apply
Proposition \ref{prop:singc2} to conclude that $(V,\omega,G)$ does not
admit a projective symplectic resolution.

\subsection{}

Now, let $G = G(K,H,\alpha) < \Sp_4(\C)$ be symplectically
irreducible.  In order for $V/G$ to admit a symplectic resolution, by
Verbitsky's theorem, $G$ must be a symplectic reflection group.  If
$G$ is not proper (cf.~\S \ref{ss:cohen}), then $G$ is a complex
reflection group, and in this case by \cite{Singular}, $V/G$ can admit
a symplectic resolution only if $G = K \wr S_2$ for $K$ a cyclic
Kleinian group, or else $G=G_4$. The latter possibility is excluded,
however, since as a subgroup of $\GL_2(\C)$ it is primitive, and hence
as a subgroup of $\Sp_4(\C)$ it is symplectically primitive.
Therefore, by \cite[Theorem 2.6]{CohenQuaternionic}, $(K,H,\alpha)$
must be on the list in Tables \ref{tab:table1} and \ref{tab:table2}.

Summarizing the above calculations and applying Theorem
\ref{thm:singularimpliesnores} therefore implies the following result.

\begin{theorem}\label{thm:rankfourclass} \label{t:main-4}
  If $G=G(K,H,\alpha)$ 
  is such that $H\neq \{1\}$ and $H \neq K$, then $(V,\omega,G)$
  admits a projective symplectic resolution if and only if $G =
  G(\biDih_2,\cyc_2,\Id)$.
\end{theorem}

Note that $G(\biDih_2,\cyc_2,\alpha_1)$ is the symplectic reflection
group $Q_8 \times_{\mathbf{Z}_2} D_8$ studied in \cite{smoothsra}.

\section{The imprimitive groups in dimension $\geq 6$}

\subsection{} As in the introduction by \cite[Theorem
2.9]{CohenQuaternionic}, the symplectically imprimitive and
irreducible symplectic reflection groups in dimensions greater than
four are, up to conjugation, of the form $G_n(K,H)$ where $K <
\SL_2(\C)$ is a Kleinian group and $H \leq K$ contains the commutator
subgroup $[K,K]$.

\begin{lemma}\label{lem:specialthreegroup}
  The symplectic reflection group $G_3(\cyc_2,\biDih_{2})$
  does not admit a projective symplectic resolution.
\end{lemma}

\begin{proof}
  The group $G_3(\cyc_2,\biDih_2) / \cyc_2^3 \simeq G_3(\{ 1 \},
  \Dih_2)$ acts on $\mathsf{Y} \times \mathsf{Y} \times \mathsf{Y}$
  where $\mathsf{Y} = T^* \mathbf{P}^1$ is the minimal resolution of
  $\C^2 / \cyc_2$. We prove that $\C^6 / G_3(\cyc_2,\biDih_{2})$ does
  not admit a projective symplectic resolution by showing that there
  exists an isolated point $p \in \mathsf{Y}^3$ whose stabilizer with
  respect to $G_3(\{ 1 \}, \Dih_2)$ is non-trivial i.e. there is an
  affine open subset $U$ of $p$ such that $G_3(\{ 1 \}, \Dih_2)$ acts
  freely on $U \backslash \{ p \}$. Repeating the argument given in
  the proof of Theorem \ref{thm:singularimpliesnores} then implies the
  claim of the lemma.

  In order to simplify things, we consider the action of the larger
  group $G' := \Dih_2 \wr S_2$ on $\mathsf{Y}^3$. As in the proof of
  Lemma \ref{lem:niceopen}, let
$$
\Delta_{1,2} = \bigcup_{h \in \Dih_2} \{ (h \cdot x, x,y) \in
\mathsf{Y}^3 \ | \ (x,y) \in \mathsf{Y}^2 \}
$$
and $\Delta = \Delta_{1,2} \cup \Delta_{1,3} \cup \Delta_{2,3}$ (where
$\Delta_{i,j}$ is defined in the obvious manner). This is a proper
closed subset of $\mathsf{Y}^3$. If there exists some $p \in
\mathsf{Y}^3$ and $g \in G' \backslash \Dih_2^3$ such that $g \cdot p
= p$ then $p \in \Delta$. To get an isolated point we need to consider
points in $\mathbf{P}^1 \times \mathbf{P}^1 \times \mathbf{P}^1$ not
contained in $\Delta$. The group $\Dih_2 \simeq \cyc_2 \times \cyc_2$
acts on $\mathbf{P}^1$ by the image of its reflection representation
in $PSL_2(\C)$. Thus, the three non-trivial elements of $\Dih_2$ are
$$
g = \left(\begin{array}{cc}
i & 0 \\
0 & -i 
\end{array} \right), \quad h = \left(\begin{array}{cc}
  0 & 1 \\
  -1 & 0 
\end{array} \right), \quad gh = \left(\begin{array}{cc}
0 & i \\
i & 0 
\end{array} \right) \ \in PSL_2(\C). 
$$
The fixed points of $g$ are $F_g = \{ [1:0],[0:1] \}$, of $h$ are $F_h
= \{ [1:i],[1:-i] \}$ and of $gh$ are $F_{gh} = \{ [1:1],[1:-1]
\}$. Each set $F_w$ is stable under the action of $\Dih_2$ with both
the non-trivial elements of $\Dih_2$ not equal to $w$ swapping the two
points in $F_w$. If we take
$$
p = ([1:0],[1:i],[1:1]) \in \mathsf{Y}^3 \backslash \Delta
$$
then, noting that $g \cdot h \cdot gh = 1$ in $PSL_2(\C)$, the
stabilizer of $p$ in $G'$ is $\{1,g \} \times \{ 1, h \} \times \{ 1 ,
gh \} \simeq \cyc_2^3$ and hence the stabilizer of $p$ in $G_3(\{ 1
\}, \Dih_2)$ is $\{ (1,1,1),(g,h,gh) \} \simeq \cyc_2$ and $p$ is
isolated.
\end{proof}

Therefore we may conclude:

\begin{theorem}\label{t:main-g4}
  Let $n > 2$. Then the symplectic quotient $\C^{2n} / G_n(K,H)$ admits a
  projective symplectic resolution if and only if $K = H$.
\end{theorem}

\begin{proof}
  If $G_n(K,H) \neq G_n(\cyc_{2},\biDih_{2})$ and $K \neq H$, then we
  choose $n-2$ distinct points $p_3, \ds, p_n$ in $\C^{2} \backslash
  \{ 0 \}$ and consider the point $(0,0,p_3,\ds, p_n)$ in
  $\C^{2n}$. The stabilizer of this point is $G_2(K,H) =
  G(K,H,\Id)$. We have shown in Theorem \ref{thm:rankfourclass} that
  all those groups $G(K,H,\alpha)$ such that $\Gamma$ is abelian
  (except for $G(\biDih_{2l},\cyc_2,\Id)$) do not admit projective
  symplectic resolutions. Now the result follows from \cite[Theorem
  1.6]{KaledinDynkin}. On the other hand, if $G_n(K,H) =
  G_n(\cyc_{2},\biDih_{2})$ with $n = 3$, then we have shown in Lemma
  \ref{lem:specialthreegroup} that the corresponding symplectic
  quotient does not admit a projective symplectic resolution. If $n >
  3$ then one can realize $G_3(\cyc_2,\biDih_{2})$ as a parabolic
  subgroup of $G_n(\cyc_{2},\biDih_{2})$ as above and the same
  argument shows that $G_n(\cyc_{2},\biDih_{2})$ does not admit a
  projective symplectic resolution.
\end{proof} 

\section{Proof of the main theorems}
\subsection{Proof of Theorem \ref{thm:main}}
Assume $G < \Sp_{2n}(\C)$ is symplectically imprimitive and
irreducible. As a consequence of Theorems \ref{t:main-4} and
\ref{t:main-g4}, $\C^{2n}/G$ cannot admit a projective symplectic
resolution if $G = G(K,H,\alpha)$ (for $n=2$) or if $G=G_n(K,H)$,
unless $G = K \wr S_n$ or $G = G(\biDih_2,\cyc_2,\Id) \simeq Q_8
\times_{\mathbf{Z}/2} D_8$.  By \cite[Theorems 2.2 and
2.9]{CohenQuaternionic}, this includes all cases where $G$ is proper.
If $G$ were not proper, then it would be a complex reflection group,
and then by \cite{Singular}, $\C^{2n}/G$ could only admit a projective
symplectic resolution if $G$ were a wreath product $(\mathbf{Z}/m) \wr
S_n$ or $G = G_4 < \Sp_4(\C)$, although the latter is excluded since
it would be symplectically primitive (as $G_4 <\GL_2(\C)$ is a
primitive complex reflection group).  Therefore, $\C^{2n}/G$ can only
admit a projective symplectic resolution if it is one of the listed
cases. On the other hand, we know that a projective symplectic
resolution exists in each of these cases. This completes the proof.

\subsection{Proof of Theorem \ref{t:four}} \label{ss:tfour}
As before, by
  \cite{Verbitsky} we can assume $G$ is a symplectic reflection group,
  and by \cite{Singular}, we can assume that $G$ is proper.  If $G <
  \Sp_2(\C)=\SL_2(\C)$, then we know that a projective symplectic
  resolution exists.  So assume $G < \Sp_{2n}(\C)$ for $n \geq 3$.  As
  explained in \S \ref{ss:cohen}, $G$ must be either of the form
  $G_n(K,H)$ or are one of the groups listed in \cite[Table
  III]{CohenQuaternionic}, where $n \leq 5$.  In the former case, the
  result follows from Theorem \ref{thm:main}. In the latter case,
  there are seven groups listed, of types $Q, R, S_1, S_2, S_3, T$,
  and $U$. The table there also lists the stabilizers in each of these
  groups of roots of the associated quaternionic root system.  For
  each group $G$ and each symplectic reflection $g \in G$, this
  stabilizer subgroup, call it $H$, is the stabilizer of generic
  vectors in the image of $g-\Id$.  The action of $H$ on the kernel of
  $g-\Id$ identifies $H$ as a subgroup of $\Sp_{2n-2}(\C)$.  By
  \cite[Theorem 1.6]{KaledinDynkin}, if $\C^{2n}/G$ admits a
  projective symplectic resolution, so does $\C^{2n-2}/H$.

  In type $Q$, we have $H =G(\cyc_4,\cyc_2,1) < \Sp_4(\C)$, and we
  showed that $\C^4/H$ does not admit a projective symplectic
  resolution in Theorem \ref{thm:main}.  Similarly, in type $S_3$, we
  have $H = G_3(\biDih_2,\cyc_2)$, which we showed does not admit a
  resolution in the same theorem (or in Lemma
  \ref{lem:specialthreegroup}). In type $T$, the group $H$ becomes a
  complex reflection group, associated to the Coxeter group of type
  $H_3$; in this cases $\C^6/H$ does not admit a projective symplectic
  resolution by the main result of \cite{Singular}.

  This reduces us to the cases $R, S_1, S_2$, and $U$, which are the
  four cases remaining, in dimensions six, eight, eight, and ten
  respectively.

\begin{remark}
  If one could show that type $S_1$ does not admit a projective
  symplectic resolution, then the same would follow for type $U$,
  since the stabilizer group $H$ mentioned above has type $S_1$. Thus,
  if the four remaining cases (as one might suspect) do not admit
  projective symplectic resolutions, it suffices only to show it for
  the three types $R, S_1$, and $S_2$.
\end{remark}

\section{Questions}

\subsection{} By definition, symplectic reflection groups are the
symplectic analogue of complex reflection groups. Therefore it is
natural to ask which properties of complex reflection groups have
natural analogues for symplectic reflection groups. In particular, one
can ask if the analogue of Steinberg's Theorem holds:

\begin{question}
  Let $(V,\omega,G)$ be a symplectic reflection group, $v \in V$ and
  $G_v = \Stab_G(v)$. Let $U$ be the symplectic complement to
  $V^{G_v}$ in $V$. Is $(U,\omega|_{U}, G_v)$ a symplectic reflection
  group?
\end{question}

\begin{remark}
  Steinberg's Theorem, together with other elementary considerations,
  show that it suffices to consider the case where $V$ is irreducible
  as a $G$-module. Furthermore, one can explicitly check for every
  complex imprimitive group that $(U,\omega|_{U}, G_v)$ is indeed a
  symplectic reflection group. Thus, it actually suffices to resolve
  the question for the complex primitive, symplectically irreducible
  symplectic reflection groups; and of these we can further restrict
  to the case of dimension at least six, since all finite subgroups of
  $\Sp_2(\C)$ are symplectic reflection groups. This narrows us down
  to checking Steinberg's theorem for the seven groups discussed in
  the previous section.  However, if it is indeed the case that the
  analogue of Steinberg's Theorem holds for symplectic reflection
  groups, it would be interesting to have a conceptual proof that does
  not rely on Cohen's classification.
\end{remark}

\subsection{} To complete the classification of symplectic reflection
groups admitting projective symplectic resolutions one needs to answer the
following three questions.

\begin{question}
  Let $G := G(K,1,\alpha) \cong K \rtimes S_2$ be a symplectically
  irreducible proper symplectic reflection group (so that $G$ belongs
  to one of the families (G),(K),(P),(Q),(U),(V) of Tables
  \ref{tab:table1} and \ref{tab:table2}). Does the quotient
  singularity $\C^4/G$ admit a projective symplectic resolution?
\end{question}

\begin{question}
Let $G$ be symplectically primitive and irreducible, but complex
imprimitive, i.e., $G < \Sp_4(\C)$ is one of the groups classified
in \cite[Lemma 3.3]{CohenQuaternionic}.  Does $\C^4/G$ admit
a projective symplectic resolution?
\end{question}

\begin{question}
  Let $G$ be one of the finitely many primitive exceptional symplectic
  reflection groups, as listed in \cite[Table
  III]{CohenQuaternionic}. Does the quotient singularity $V/G$ admit a
  projective symplectic resolution?
\end{question}

It seems likely that many of these exceptional groups $G$ will contain
parabolic subgroups $G_v$ such that $(U,\omega|_{U}, G_v)$ is known
not to admit projective symplectic resolutions. In these cases
$(V,\omega,G)$ will also not admit a projective symplectic resolution.
In particular, using the stabilizer groups discussed in \S
\ref{ss:tfour}, as we mentioned, this already allows us to eliminate
types $Q, S_3$, and $T$.  Thus there remain at most ten groups in
\cite[Table III]{CohenQuaternionic} to check.


\def\cprime{$'$} \def\cprime{$'$} \def\cprime{$'$} \def\cprime{$'$}
  \def\cprime{$'$} \def\cprime{$'$} \def\cprime{$'$} \def\cprime{$'$}

\end{document}